\documentclass[a4paper,10pt]{article}

\usepackage{amsfonts}
\usepackage{amsthm}
\usepackage{amsmath}
\usepackage{amssymb}
\usepackage{mathtools} 

\usepackage[top=1in, bottom=1in, left=1in, right=1in]{geometry}
\usepackage{xcolor}

\usepackage{graphicx}
\usepackage{epstopdf}
\usepackage{subfigure}

\usepackage{authblk}

\theoremstyle{plain}
\newtheorem{theorem}{Theorem}[section]
\newtheorem{corollary}[theorem]{Corollary}
\newtheorem{lemma}[theorem]{Lemma}

\theoremstyle{definition}
\newtheorem{definition}[theorem]{Definition}

\theoremstyle{remark}

\usepackage{xcolor}
\usepackage{soul}
\definecolor{afcol}{rgb}{1,0,0}

\providecommand{\keywords}[1]{\textbf{\textit{Keywords:}} #1}

\begin{document}

\title{Lipschitz and Fourier type conditions with moduli of continuity in rank 1 symmetric spaces}

\date{}

\author[1]{Arran Fernandez\thanks{Email: \texttt{arran.fernandez@emu.edu.tr}}}
\author[2,3]{Joel E. Restrepo\thanks{Email: \texttt{cocojoel89@yahoo.es}; \texttt{joel.restrepo@nu.edu.kz}}}
\author[2]{Durvudkhan Suragan\thanks{Email: \texttt{durvudkhan.suragan@nu.edu.kz}}}

\affil[1]{{\small Department of Mathematics, Faculty of Arts and Sciences, Eastern Mediterranean University, Famagusta, Northern Cyprus, via Mersin 10, Turkey}}
\affil[2]{{\small Department of Mathematics, Nazarbayev University, Nur-Sultan, Kazakhstan}}
\affil[3]{{\small Institute of Mathematics, University of Antioquia, Medellin, Colombia}}

\maketitle

\begin{abstract}
Sufficient and necessary results have been proven on Lipschitz type integral conditions and bounds of its Fourier transform for an $L^2$ function, in the setting of Riemannian symmetric spaces of rank $1$ whose growth depends on a $k$th-order modulus of continuity.  
\end{abstract}

\keywords{generalised H\"older space; Lipschitz type condition; Fourier transform; moduli of continuity; translation operator; symmetric space}

\section{Introduction}

The inspiration for this research begins with the classical theorem of Titchmarsh \cite[Theorem 85]{ti} on the characterisation of functions in $L^2(\mathbb{R})$ satisfying an integral Lipschitz-type condition in terms of an asymptotic estimate for the growth of the norm of their Fourier transform. This theorem is stated as follows.

\begin{theorem}[Titchmarsh \cite{ti}] \label{Tit}
Let $f\in L^2(\mathbb{R})$ and $0<\alpha<1$. Then the following statements are equivalent.
\begin{enumerate}
\item There exists a constant $C>0$ such that for all sufficiently small $t>0$,
\[\int_{-\infty}^{\infty}\left|f(x+t)-f(x-t)\right|^2\,\mathrm{d}x\leq Ct^{2\alpha}.\]
\item There exists a constant $C>0$ such that for all sufficiently small $t>0$,
\[\left(\int^{-1/t}_{-\infty}+\int_{1/t}^{+\infty}\right)\left|\widehat{f}(x)\right|^2\,\mathrm{d}x\leq Ct^{2\alpha}.\]
\end{enumerate}
\end{theorem}

This result has been shown to be a very useful tool to give an alternative description of the Lipschitz or H\"older spaces, see e.g. \cite{daher2019,platonov} and references therein. Similar results using moduli of continuity for the characterisation of Lipschitz conditions by means of the Fourier transform can be found in \cite{2008,intro1,intro5} among others.     

Following mainly the ideas from \cite{2008,platonov,ti}, we study and establish growth conditions between functions in $L^2(\mathbb{R})$ satisfying an integral Lipschitz-type condition and its Fourier transforms by means of $k$th-order moduli of continuity over rank $1$ Riemannian symmetric spaces. Some examples of the latter spaces are the two-point homogeneous spaces \cite[Chapter 8]{5pla}, the $n$-dimensional sphere $\mathbb{S}^n$, and the $n$-dimensional Lobachevski\^i space. The results are in the form of sufficient and necessary conditions (under certain assumptions) for Lipschitz-type conditions using asymptotic estimates for the growth of the Fourier transform. In a very particular case, an analogue of Titchmarsh's original Theorem \ref{Tit} is recovered.
 
The paper has the following structure. Section \ref{preli} is devoted to setting up all the necessary framework for the problem: firstly all the machinery of Riemannian symmetric spaces, and secondly the essential definitions and properties of $k$th-order moduli of continuity. In Section \ref{mainre}, the main results are stated and proved, with all required conditions added for both the necessary and the sufficient condition. Finally, Section \ref{scsection} shows special cases and analogous results on the Euclidean space $\mathbb{R}^n$.

\section{Preliminaries}\label{preli}

This section is for recalling various definitions and facts which will be used later in proving the main results of the paper. Firstly, a discussion of semi-simple Lie groups and symmetric spaces with particular emphasis on the Fourier transform; for further information about these topics, the reader is recommended to \cite{2pla,3pla}. Secondly, a brief introduction to moduli of continuity and $k$th-order moduli of continuity will be given.

\subsection{Riemannian symmetric spaces}

A noncompact Riemannian symmetric space $X$ can be seen as a quotient space $G/K$ where $G$ is a connected noncompact semi-simple Lie group with finite center and $K$ is a maximal compact subgroup of $G$. Then $G$ acts transitively on $G/K$ by left translations. Moreover, $K$ coincides with the stationary subgroup of the point $o=eK\in X$, where $e$ is the identity of $G$.

The Iwasawa decomposition of the group $G$ is written $G=NAK$, where $N$, $A$, $K$ are the Lie subgroups of $G$ generated by the Lie algebras $\mathfrak {n}_{0}$, $\mathfrak {a}_{0}$ and $\mathfrak {k}_{0}$, respectively. Let $M$ be the centraliser of $A$ in $K$ and set $B=K/M$. The $G$-invariant measure on $X$ will be denoted $dx$, and the normalised $K$-invariant measure on $B$ and $K$ will be denoted $db$ and $dk$ respectively.

Assume the symmetric space $X$ has rank $1$, so that the real dual $\mathfrak{a}_0^*$ of the Lie algebra $\mathfrak{a}_0$ is 1-dimensional and can be identified isomorphically with $\mathbb{R}$. Let $W$ denote the finite Weyl group acting on $\mathfrak{a}_0^*$. Let $\Sigma\subset\mathfrak{a}_0^*$ be the set of all bounded roots, and $\Sigma^{+}$ be the subset of positive bounded roots. For any $\lambda\in\mathfrak{a}^*_0$, say $T_\lambda$ is the vector in $\mathfrak{a}_0$ such that $\lambda(T)=\langle T_\lambda,T\rangle$ for any $T\in\mathfrak{a}_0$ (here $\langle,\rangle$ is the Killing form on the Lie algebra $\mathfrak{g}$ of $G$). Then define $\mathfrak{a}_0^+=\left\{r\in\mathfrak{a}_0:\alpha(r)>0\,\,\text{for}\,\,\alpha\in\Sigma^+\right\}$ and $\mathfrak{a}_{0,+}^*=\left\{\lambda\in\mathfrak{a}_0^*:T_\lambda\in\mathfrak{a}_0^{+}\right\}$. Under the identification of $\mathfrak{a}_0^*$ with $\mathbb{R}$, it follows that $\mathfrak{a}_{0,+}^*$ can be identified with $\mathbb{R}^+$. For more details on this, see e.g. \cite{pla1999,platonov}.

For any $g\in G$, taking into account that $G=NAK$, there is a unique element $A(g)\in \mathfrak{a}_0$ such that 
\[g=n\cdot \text{exp}\,A(g)\cdot k,\quad n\in N,\,\,k\in K.\]       
The map $A:G\to \mathfrak{a}_0$ thus defined gives rise to a well-defined map $A:X\times B\to \mathfrak{a}_0$ as follows:
\[A(x,b):=A(k^{-1}g),\quad x=gK\in X=G/K,\,\, b=kM\in B=K/M.\]
Let $\mathcal{D}(X)$ and $\mathcal{D}(G)$ denote the sets of all infinitely differentiable compactly-supported functions on $X$ and $G$ respectively. Let $dg$ be the Haar measure on $G$, normalised so that
\[\int_X f(x)dx=\int_G f(go)dg,\quad f\in\mathcal{D}(X).\]
The Fourier transform on $X$ is defined by
\begin{equation}
\label{Fourier}
\widehat{f}(\lambda,b)=\int_X f(x)e^{(-i\lambda+\rho)A(x,b)}dx,\, \lambda\in\mathfrak{a}_0^*,\,b\in B=K/M,\,\rho=\frac12\sum_{\alpha\in\Sigma^+}\alpha,
\end{equation}
and the inverse Fourier transform on $X$ is defined by
\[f(x)=\frac{1}{|W|}\int_{\mathfrak{a}_0^*\times B}\widehat{f}(\lambda,b)e^{(i\lambda+\rho)A(x,b)}|c(\lambda)|^{-2}d\lambda db,\]
where $|W|$ is the order of the Weyl group, $d\lambda$ is the Euclidean measure on $\mathfrak{a}_0^*$ (identified with $\mathbb{R}$), and $c(\lambda)$ is the Harish-Chandra function. For more details see e.g. \cite{3pla,8pla}.

The Harish-Chandra formula defines a so-called spherical function $\varphi_{\lambda}$ on $G$ for any $\lambda\in\mathfrak{a}_0^*$ as follows:
\begin{equation}\label{pla2.9} 
\varphi_\lambda(g)=\int_K e^{(i\lambda+\rho)A(kg)}dk,\quad g\in G.
\end{equation}
One of its properties is that $\varphi_\lambda(k_1gk_2)=\varphi_\lambda(g)$ for any $k_1,k_2\in K$ and $g\in G$. Using this property it can be proved that $\varphi_\lambda(g)$ depends only on the distance $d(go,o)=t$ which is a non-negative real number. So, where convenient, $\varphi_\lambda(t)$ can be written instead of $\varphi_\lambda(g)$. See the proof of \cite[Lemma 3]{platonov} for more details.

The following lemma on some estimates of the function $\varphi_\lambda$ will be very useful later.

\begin{lemma}\cite[Lemmas 3.1--3.3]{pla1999}\label{muno}
For any $\lambda\in\mathbb{R}$ (identified with $\mathfrak{a}_0^*$) and any $t\in\mathbb{R}^+_0$:
\begin{enumerate}
\item $|\varphi_\lambda(t)|\leq1$, with equality just at the point $t=0$, i.e. $|\phi_\lambda(0)|=1$.
\item $1-\varphi_\lambda(t)\leq t^2(\lambda^2+\rho^2)$.
\item There exists a constant $C>0$, independent of $\lambda$ and $t$, such that $1-\varphi_\lambda(t)\geq C$ whenever $\lambda\geq1/t.$
\end{enumerate}
\end{lemma}

The abstract Fourier transform on $X$ defined by \eqref{Fourier} satisfies a Plancherel formula on $X$ just like the classical Fourier transform:
\[\int_X |f(x)|^2 dx=\frac{1}{|W|}\int_{\mathfrak{a}_0^*\times B}|\widehat{f}(\lambda,b)|^2 d\mu(\lambda)db=\int_{\mathfrak{a}_{0,+}^*\times B}|\widehat{f}(\lambda,b)|^2 d\mu(\lambda)db,\]
where $d\mu(\lambda)=|c(\lambda)|^{-2}d\lambda$ denotes the measure used on $\mathfrak{a}_0^*$. Therefore, the continuous mapping $f\to \widehat{f}$ can be considered as an isomorphism from the space $L^2(X,dx)$ onto the space $L^2(\mathfrak{a}_{0,+}^*\times B,d\mu(\lambda)db)$. For more details about this, see \cite{gplan1,gplan2} and also \cite[Formula (1.9)]{platonov}.

The results of this paper will be concerned with the Fourier transform defined in \eqref{Fourier} above, and also with the translation operator defined in \eqref{toperator} below. The latter is essentially defined by taking averages of a function $f$ on sphere-like neighbourhoods. So, first it is necessary to define the framework for such neighbourhoods.

For any $x\in X$ and $t\in\mathbb{R}^+$, the sphere of radius $t$ in $X$ centred at $x$ is denoted by
\[\sigma(x;t)=\{y\in X: d(x,y)=t\},\]
where as usual $d(x,y)$ denotes the distance between the elements $x,y\in X$. This has an associated area element $d\sigma_x(y)$ used for integration, and overall surface area $|\sigma(t)|$ of the whole sphere. The translation operator $S^t$ is defined as follows, for functions $f\in C_0(X)$, the set of all continuous compactly-supported functions on $X$:
\begin{equation}\label{toperator}
(S^t f)(x)=\frac{1}{|\sigma(t)|}\int_{\sigma(x;t)}f(y)d\sigma_x(y),\quad t>0.
\end{equation}     
It can be proved that the operator $S^t$ is bounded from $C_0(X)$ to $L^2(X)$, see \cite[Lemma 2]{platonov} and consequences. In \cite[Lemma 3]{platonov}, it was shown that the translation operator, spherical function, and Fourier transform have the following interrelationship:
\begin{equation}
\label{TSFrelation}
\widehat{S^t f}(\lambda,b)=\varphi_{\lambda}(t)\widehat{f}(\lambda,b),\quad f\in L^2(X), t\in\mathbb{R}^+_0.
\end{equation}

\subsection{Moduli of continuity of $k$th order}

\begin{definition}
The function $\omega:I\subset\mathbb{R}\rightarrow[0,\infty)$ is called almost increasing if there exists a constant $C\geq 1$ such that $\omega(t)\leq C\omega(s)$ for all $t,s\in I$ with $t\leq s$. Moreover, it is called almost decreasing if there exists a constant $C\geq 1$ such that $\omega(t)\leq C\omega(s)$ for all $t,s\in I$ with $t\geq s$.
\end{definition}
\begin{definition}
Let $\delta_0\in(0,\infty)$ be a fixed real number and $k\in\mathbb{R}^+$. The function $\omega_k:[0,\delta_0]\to[0,\infty)$ is called a $k$th-order modulus of continuity if the following conditions hold:
\begin{enumerate}
\item\label{m1} $\omega_k(0)=0$ and $\omega_k(t)$ is continuous on $[0,\delta_0]$.
\item\label{m2} $\omega_k(t)$ is almost increasing on $t\in[0,\delta_0]$.
\item\label{m3} $\frac{\omega_k(t)}{t^k}$ is almost decreasing on $t\in[0,\delta_0]$.
\end{enumerate}
\end{definition}

It is important to note the following fact. Any $k$th-order modulus of continuity is also an $m$th-order modulus of continuity for all $m\geq k$: conditions \eqref{m1} and \eqref{m2} are the same, and dividing an almost decreasing function by a positive power of $t$ gives another almost decreasing function. This fact can be illustrated by considering power functions: for any given $k$, the function $\omega(t)=t^{\gamma}$ is a $k$th-order modulus of continuity whenever $0<\gamma<k$.

The following Zygmund type conditions of $k$th order will also frequently be used:
\begin{enumerate}
\item There exists a constant $C$ such that for all $t\in[0,\delta_0]$,
\begin{equation}\label{cond4}
\int_0^t\frac{\omega_k(x)}{x}\,\mathrm{d}x\leq C\omega_k(t).
\end{equation}
\item There exists a constant $C$ such that for all $t\in[0,\delta_0]$,
\begin{equation}\label{cond3}
\int_t^{\delta_0}\frac{\omega_k(x)}{x^{1+k}}\,\mathrm{d}x\leq C\frac{\omega_k(t)}{t^k}.
\end{equation}
\end{enumerate}
In the book \cite[Definition 2.9]{rafeirobook}, the notation $Z_k$ is used for functions satisfying \eqref{cond3}, and the condition \eqref{cond4} says that $\omega\in Z^0$ (in general $Z^{\beta}$ would be defined by replacing the integrand and right-hand side in \eqref{cond4} with those from \eqref{cond3}).

The Matuszewska-Orlicz type (MO) lower and upper indices of the function $\omega_k(t)$ is also related to these function spaces. The definitions are \cite{rafeirobook}:
\begin{align*}
m(\omega_k)&=\sup_{0<t<1}\frac{\log\left(\displaystyle\limsup_{\varepsilon\to0}\frac{\omega_k(\varepsilon t)}{\omega_k(\varepsilon)}\right)}{\log t}=\lim_{t\to0}\frac{\log\left(\displaystyle\limsup_{\varepsilon\to0}\frac{\omega_k(\varepsilon t)}{\omega_k(\varepsilon)}\right)}{\log t}, \\
M(\omega_k)&=\sup_{t>1}\frac{\log\left(\displaystyle\limsup_{\varepsilon\to0}\frac{\omega_k(\varepsilon t)}{\omega_k(\varepsilon)}\right)}{\log t}=\lim_{t\to\infty}\frac{\log\left(\displaystyle\limsup_{\varepsilon\to0}\frac{\omega_k(\varepsilon t)}{\omega_k(\varepsilon)}\right)}{\log t},
\end{align*}
and these are related to the Zygmund conditions and spaces by the following result \cite[Theorem 2.10]{rafeirobook}:
\begin{theorem}\label{rafeiro}
Let $\omega_k$ be a $k$th-order modulus of continuity as above, and $\delta\in\mathbb{R}$. Then $\omega_k\in Z^{\delta}$ if and only if $m(\omega_k)>\delta$, while $\omega_k\in Z_{\delta}$ if and only if $M(\omega_k)<\delta$. Moreover, we have 
\begin{align*}
m(\omega_k)=\sup\left\{\delta>0:\frac{\omega_k(t)}{t^{\delta}}\,\,\text{is almost increasing}\right\}, \\
M(\omega_k)=\inf\left\{\delta>0:\frac{\omega_k(t)}{t^{\delta}}\,\,\text{is almost decreasing}\right\}.
\end{align*}
\end{theorem}

If $\omega$ is a modulus of continuity and satisfies the extra Zygmund type conditions \eqref{cond4} and \eqref{cond3} then $\omega$ belongs to the so-called Bary--Stechkin class \cite{nata111}. Some classical examples of functions in the Bary--Stechkin class are $t^{\gamma}$, $t^{\gamma}\left(\log\frac1{t}\right)^{\lambda}$, and $t^{\gamma}\left(\log\log\frac1{t}\right)^{\lambda}$, where $\gamma\in(0,1)$ and $\lambda\in\mathbb{R}$. These examples can be adapted easily to the $k$th-order case, simply replacing the condition $\gamma\in(0,1)$ by the more general condition $\gamma\in(0,k)$.  

The most important behaviour of the functions above is near zero. But, when it is necessary to do some estimations on these functions over $[\delta_0,\infty)$, it will be assumed that $\omega_k(t)$ is bounded below on $[\delta_0,\infty)$ by a positive number, and that $\omega_k(t)^2/t^{5}\in L^1\big([\delta_0,\infty)\big)$, without loss of generality of the results.

\section{Main results}\label{mainre}

This section begins by introducing the space to be considered in this paper.

\begin{definition}
Let $\omega_k$ be a $k$th-order modulus of continuity. Let $X$ be a Riemannian symmetric space of rank $1$ with $n=\dim X$. A function $f\in L^2(X)$ is said to be in the generalised Lipschitz class $\text{Lip}(\omega_k)$ if there exists a constant $C>0$ such that for all sufficiently small $t>0$,
$$ \|S^t f-f\|_{L^2(X)}\leq C\omega_k(t). $$
\end{definition}

Below is one of the main results of this paper, an estimate of the Fourier transform of any function in $\text{Lip}(\omega_k)$, or in other words a necessary condition in terms of Fourier transforms for a function to be in this generalised Lipschitz class.

\begin{theorem}\label{main1}
Let $\omega_k$ be a $k$th-order modulus of continuity. Let $X$ be a Riemannian symmetric space of rank $1$ with $n=\dim X$.  If $f\in \text{Lip}(\omega_k)$, then there exists a constant $C>0$ such that for all sufficiently small $t>0$,
\begin{equation}\label{pla1.13}
\int_{1/t}^{+\infty}\int_{B}|\widehat{f}(\lambda,b)|^2 d\lambda db\leq C\omega_k(t)^2t^{n-1}. 
\end{equation}
\end{theorem}

\begin{proof}
Notice that \cite[Formula (3.2)]{platonov}, which follows from the Plancherel formula and the identity \eqref{TSFrelation}, gives the following: 
\begin{equation}\label{pla3.2}
\|S^t f-f\|_{L^2(X)}^2=\int_X \big|S^t f(x)-f(x)\big|^2 dx=\int_{0}^{+\infty}|1-\varphi_\lambda(t)|^2 H(\lambda)d\mu(\lambda),
\end{equation}
where 
\[H(\lambda)=\int_B |\widehat{f}(\lambda,b)|^2 db\quad\text{and}\quad d\mu(\lambda)=|c(\lambda)|^{-2}d\lambda.\]
By the representation \eqref{pla3.2}, part 3 of Lemma \ref{muno}, and the fact that $f\in \text{Lip}(\omega_k)$, we have
\begin{align}\label{newuf}
\int_{1/t}^{+\infty}H(\lambda)d\mu(\lambda)&\leq\frac{1}{C_1^2} \int_{1/t}^{+\infty}|1-\varphi_\lambda(t)|^2 H(\lambda)d\mu(\lambda)\nonumber \\
&\leq\frac{1}{C_1^2}\int_{0}^{+\infty}|1-\varphi_\lambda(t)|^2 H(\lambda)d\mu(\lambda) \leq C_2\omega_k(t)^2.
\end{align}
Besides, by \cite[formula (3.6)]{platonov} it is known that  
\begin{equation}\label{pla3.6}
|c(\lambda)|^{-2}\approx \lambda^{n-1}.
\end{equation}
Therefore, by \eqref{newuf}, \eqref{pla3.6} and $\lambda\geq1/t$ it follows that
\begin{align*}
\int_{1/t}^{+\infty}\int_{B}&|\widehat{f}(\lambda,b)|^2 d\lambda db=\int_{1/t}^{+\infty}H(\lambda)d\lambda=t^{n-1}\int_{1/t}^{+\infty}H(\lambda)\left(\frac{1}{t}\right)^{n-1}d\lambda \\
&\leq t^{n-1}\int_{1/t}^{+\infty}H(\lambda)\lambda^{n-1}d\lambda\approx t^{n-1}\int_{1/t}^{+\infty}H(\lambda)d\mu(\lambda)\leq C_3t^{n-1}\omega_k(t)^2.
\end{align*}
\end{proof}

For the converse statement of Theorem \ref{main1}, it is necessary to add an extra assumption on $k$. The result is given by the following theorem.

\begin{theorem}\label{main2}
Let $k\leq2$, and let $\omega_k$ be a modulus of continuity of order $k$, satisfying the Zygmund conditions \eqref{cond4} and \eqref{cond3}. It is also assumed that $\omega_k(t)$ is bounded below on $[\delta_0,\infty)$ by a positive number, and that $\omega_k(t)^2/t^{5}\in L^1\big([\delta_0,\infty)\big)$. Let $X$ be a Riemannian symmetric space of rank $1$ with $n=\dim X$. If there exists a constant $C>0$ such that for all sufficiently small $t>0$ the condition \eqref{pla1.13}, then $f\in \text{Lip}(\omega_k)$.
\end{theorem}

\begin{proof}
With the same notation $H(\lambda)$ as used in the proof of Theorem \ref{main1}, the condition \eqref{pla1.13} implies
\[\int_{1/t}^{2/t}H(\lambda)\lambda^{n-1}d\lambda\leq\left(\frac{2}{t}\right)^{n-1}\int_{1/t}^{2/t}H(\lambda)d\lambda\leq\left(\frac{2}{t}\right)^{n-1}\int_{1/t}^{+\infty}H(\lambda)d\lambda\leq C_1\omega_k(t)^2.\]
Hence, substituting $t/2^j$ for $t$ and then combining the results for different $j$,
\begin{align*}
\int_{1/t}^{+\infty}H(\lambda)\lambda^{n-1}d\lambda&=\sum_{j=0}^{+\infty}\int_{2^{j}/t}^{2^{j+1}/t}H(\lambda)\lambda^{n-1}d\lambda\leq C_1\sum_{j=0}^{+\infty}\omega_k\left(\frac{t}{2^j}\right)^2 \\
&\leq C_2\frac{\omega_k(t)^2}{t^{2\delta}}\sum_{j=0}^{+\infty}\left(\frac{t}{2^j}\right)^{2\delta}\leq C\omega_k(t)^2,  
\end{align*}
where $\delta>0$ is less than the MO index $m(\omega)$, so that $\frac{\omega(t)}{t^{\delta}}$ is almost increasing by Theorem \ref{rafeiro}. Consequently, using \eqref{pla3.6}, the following is obtained:
\begin{equation}\label{pla3.14}
\int_{1/t}^{+\infty}H(\lambda)d\mu(\lambda)\leq C_3\omega_k(t)^2.
\end{equation}
The integral on the right-hand side of \eqref{pla3.2} can be split as follows:
\[\|S^t f-f\|_{L^2(X)}^2=J_1+J_2,\]
where 
\[J_1=\int_{0}^{1/t}|1-\varphi_\lambda(t)|^2 H(\lambda)d\mu(\lambda),\quad\quad J_2=\int_{1/t}^{+\infty}|1-\varphi_\lambda(t)|^2 H(\lambda)d\mu(\lambda).\]
The second term $J_2$ can be estimated using \eqref{pla3.14} and the first part of Lemma \ref{muno}:
\[J_2\leq 4\int_{1/t}^{+\infty}H(\lambda)d\mu(\lambda)\leq C_3\omega_k(t)^2.\]
For the first term $J_1$, use the second part of Lemma \ref{muno}:
\[J_1\leq C_4t^4 \int_{0}^{1/t}(\lambda^4+\rho^4) H(\lambda)d\mu(\lambda)= K_1+K_2,\]
where
\[K_1=C_4t^4 \int_{0}^{1/t}\lambda^4 H(\lambda)d\mu(\lambda),\quad\quad K_2=C_4t^4 \rho^4 \int_{0}^{1/t}H(\lambda)d\mu(\lambda).\]
Here the second term $K_2$ can be estimated using Plancherel's theorem:
\[K_2\leq C_4t^4 \rho^4 \int_{0}^{+\infty}H(\lambda)d\mu(\lambda)\leq C_4t^4 \rho^4 \|f\|_2^2 \leq C_5\omega_k(t)^2,\]
since $t^2\leq C\omega_k(t)$. Here the assumption $k\leq2$ is used. It remains only to estimate $K_1$.

The formula \eqref{pla3.14} is assumed to hold for all sufficiently small $t>0$. But, by the Plancherel formula and the fact that $\omega_k$ is bounded below on any closed interval in the positive reals, it can equally well be assumed that \eqref{pla3.14} holds for all $t\in\mathbb{R}^+$. Using the latter fact and setting $\phi(t)=\int_{t}^{+\infty}H(\lambda)\,\mathrm{d}\mu(\lambda)$ for any $t>0$, the estimating of $K_1$ can be done as follows:
\begin{align*}
K_1&=C_4t^4\int_{0}^{1/t} (-s^4\phi^{\prime}(s))\,\mathrm{d}s=C_4t^4\left(-\frac{1}{t^4}\phi\left(\frac{1}{t}\right)+4\int_0^{1/t}s^3\phi(s)\,\mathrm{d}s\right) \\
&\leq 4C_4t^{4}\int_0^{1/t}s^3\phi(s)\,\mathrm{d}s\leq C_5t^4 \int_0^{1/t} s^3\omega_k(1/s)^2\,\mathrm{d}s=C_5t^4 \int_{t}^{+\infty}\frac{\omega_k(u)^2}{u^5}\,\mathrm{d}u \\
&=C_5t^4\left(\int_{t}^{\delta_0}\frac{\omega_k(u)^2}{u^5}\,\mathrm{d}u+\int_{\delta_0}^{+\infty}\frac{\omega_k(u)^2}{u^5}\,\mathrm{d}u\right) \\
&\leq C_6\left(t^4\cdot\frac{\omega_k(t)}{t^{4-k}}\int_{t}^{\delta_0}\frac{\omega_k(u)}{u^{k+1}}\,\mathrm{d}u+t^4\right)\leq C\omega_k(t)^2,
\end{align*}
where in the last two lines the following were used: the fact that $\omega_k(t)/t^{4-k}$ is almost decreasing (since $4-k\geq k$ because $k\leq2$), the assumption $\omega_k(t)^2/t^5\in L^1[\delta_0,+\infty)$, the Zygmund condition \eqref{cond3}, and the fact that $t^2\leq C\omega_k(t)$ since $k\leq2$.

So finally it is established that $\|S^t f-f\|_{L^2(X)}\leq C\omega_k(t).$
\end{proof}

Combining Theorem \ref{main1} and Theorem \ref{main2} gives the following equivalence relation between bounding conditions on the translation operator and on the Fourier transform.

\begin{theorem} \label{main3}
Let $X$ be a Riemannian symmetric space of rank $1$ with $n=\dim X$. Let $\omega_k$ be a $k$th-order modulus of continuity, with $k\leq2$, satisfying the Zygmund conditions \eqref{cond4} and \eqref{cond3}. It is also assumed that $\omega_k(t)$ is bounded below on $[\delta_0,\infty)$ by a positive number, and that $\omega_k(t)^2/t^{5}\in L^1\big([\delta_0,\infty)\big)$. If $f$ is a function in $L^2(X)$, then the following statements are equivalent.
\begin{enumerate}
\item\label{1} There exists a constant $C>0$ such that for all sufficiently small $t>0$,
\[ \|S^t f-f\|_{L^2(X)}\leq C\omega_k(t). \]
\item\label{2} There exists a constant $C>0$ such that for all sufficiently small $t>0$,
\[\int_{1/t}^{+\infty}\int_{B}|\widehat{f}(\lambda,b)|^2\,\mathrm{d}\lambda\,\mathrm{d}b\leq C\omega_k(t)^2t^{n-1}.\]
\end{enumerate}
\end{theorem}

\begin{proof}
Both directions of the equivalence are proved in Theorems \ref{main1}, \ref{main2} respectively.
\end{proof}

As an example of the above results, consider the function $\omega(t)=t^{\alpha}$ for any $0<\alpha<2$ and $t\in[0,\delta_0]$. This is a $2$nd-order modulus of continuity ($k=2$). To ensure the appropriate conditions on $[\delta_0,+\infty)$, assume that $\omega(t)\equiv W$ for any $t\geq \delta_0$ and some constant $W>0$. Then the following result is obtained as a corollary of Theorem \ref{main3}.

\begin{corollary} \label{cor1}
Let $X$ be a Riemannian symmetric space of rank $1$ with $n=\dim X$, and let $0<\alpha<2$. If $f$ is a function in $L^2(X)$, then the following assertions are equivalent:
\begin{enumerate}
\item There exists a constant $C>0$ such that for all sufficiently small $t>0$,
\[ \|S^t f-f\|_{L^2(X)}\leq Ct^\alpha. \]
\item There exists a constant $C>0$ such that for all sufficiently small $t>0$,
\[\int_{1/t}^{+\infty}\int_{B}|\widehat{f}(\lambda,b)|^2 \,\mathrm{d}\lambda\,\mathrm{d}b\leq Ct^{2\alpha+n-1}.\]
\end{enumerate}
\end{corollary}

\section{Special cases on $\mathbb{R}^n$}\label{scsection}

The symmetric spaces considered in this paper and the Euclidean spaces $\mathbb{R}^n$ ($n\geq1$) belong to the noncompact two-point homogeneous spaces \cite{5pla}. Thus, similar results as those given in Section \ref{mainre} can be established for the space $\mathbb{R}^n$. 

The following notation will be used. The usual inner product and norm on $\mathbb{R}^n$ are given by $\langle x,y\rangle=x_1 y_1+\cdots+x_n y_n,$ $|x|=\sqrt{\langle x,x\rangle},$ for any $x=(x_1,\ldots,x_n)$, $y=(y_1,\ldots,y_n)$ in $\mathbb{R}^n$. The Fourier transform on $\mathbb{R}^n$ is defined as 
\[\widehat{f}(y)=\frac{1}{(2\pi)^{n/2}}\int_{\mathbb{R}^n}f(x)e^{-i\langle y,x\rangle}\,\mathrm{d}x,\quad y\in\mathbb{R}^n,\]
for any $f\in C_0(\mathbb{R}^n)$. Moreover, by continuity the Fourier transform can be extended to the Hilbert space $L^2(\mathbb{R}^n)$. 

Let $\sigma$ be the unit sphere on $\mathbb{R}^n$, with $dx$ denoting the $(n-1)$-dimensional area element of the sphere and $|\sigma|$ the hypersurface area of the whole sphere. For $f\in C_0(\mathbb{R}^n)$ the translation operator $S^t$ is defined as  
\begin{equation}\label{toperatorp}
(S^t f)(x)=\frac{1}{|\sigma|}\int_{\sigma}f(x+ty)\,\mathrm{d}y,\quad t\geq 0.
\end{equation}     
This operator is also called the spherical mean operator. By continuity the operator $S^t$ can be extended to the Hilbert space $L^2(\mathbb{R}^n)$. 

Let $\omega_k$ be a $k$th-order modulus of continuity. A function $f\in L^2(\mathbb{R}^n)$ is said to be in the spherical Lipschitz class $\text{Lip}^s_{\mathbb{R}^n}(\omega_2)$ if there exists a constant $C>0$ such that for all sufficiently small $t>0$,
\[\|S^t f-f\|_{L^2(\mathbb{R}^n)}\leq C\omega_2(t).\]

The proof of the next results follows similarly as those given in Theorems \ref{main1}, \ref{main2}, and the overall equivalence condition of Theorem \ref{main3}.

\begin{theorem}\label{main1sc}
Let $\omega_k$ be a $k$th-order modulus of continuity. It is also assumed that $\omega_k(t)$ is bounded below on $[\delta_0,\infty)$ by a positive number, and that $\omega_k(t)^2/t^{5}\in L^1\big([\delta_0,\infty)\big)$. If $f\in \text{Lip}^s_{\mathbb{R}^n}(\omega_k)$, then there exists a constant $C>0$ such that for all sufficiently small $t>0$,
\begin{equation}\label{pla1.13sc}
\int_{1/t}^{+\infty}\int_{\sigma}|\widehat{f}(\lambda x)|^2\,\mathrm{d}\lambda\,\mathrm{d}x\leq C\omega_k(t)^2t^{n-1}. 
\end{equation}
\end{theorem}

\begin{theorem}\label{main2sc}
Let $\omega_k$ be a modulus of continuity of order $k\leq2$ satisfying the conditions \eqref{cond4} and \eqref{cond3}. It is also assumed that $\omega_k(t)$ is bounded below on $[\delta_0,\infty)$ by a positive number, and that $\omega_k(t)^2/t^{5}\in L^1\big([\delta_0,\infty)\big)$. If $f\in L^2(\mathbb{R}^n)$ and there exists a constant $C>0$ such that for all sufficiently small $t>0$ the condition \eqref{pla1.13sc} holds, then $f\in \text{Lip}^s_{\mathbb{R}^n}(\omega_k)$.
\end{theorem}


\begin{theorem}\label{rere2}
If $\omega_k$ is a $k$th-order modulus of continuity, for $k\leq2$, satisfying the Zygmund conditions \eqref{cond4} and \eqref{cond3}, and $\omega_k(t)$ is bounded below on $[\delta_0,\infty)$ by a positive number, and that $\omega_k(t)^2/t^{5}\in L^1\big([\delta_0,\infty)\big)$. If $f$ is a function in $L^2(\mathbb{R}^n)$, then the following statements are equivalent.
\begin{enumerate}
\item There exists a constant $C>0$ such that for all sufficiently small $t>0$,
\[ \|S^t f-f\|_{L^2(\mathbb{R}^n)}\leq C\omega_2(t). \]
\item There exists a constant $C>0$ such that for all sufficiently small $t>0$,
\[\int_{1/t}^{+\infty}\int_{\sigma}|\widehat{f}(tx)|^2 dtdx\leq C\omega_k(t)^2t^{n-1}.\]
\end{enumerate}
\end{theorem}

The above is an analogue of the classical theorem of Titchmarsh \cite[Theorem 85]{ti}, extended to $n$-dimensional space and with bounds given by generalised $k$th-order moduli of continuity rather than just power functions. The following corollary, given like Corollary \ref{cor1} above by the special case $\omega(t)=t^{\alpha}$ with $0<\alpha<2$, looks even more similar to the theorem of Titchmarsh, but still taking place in $n$-dimensional space.

\begin{corollary}
For $\alpha\in(0,2)$ and $f\in L^2(\mathbb{R}^n)$, the following assertions are equivalent.
\begin{enumerate}
\item There exists a constant $C>0$ such that for all sufficiently small $t>0$,
\[ \|S^t f-f\|_{L^2(\mathbb{R}^n)}\leq Ct^{\alpha}. \]
\item There exists a constant $C>0$ such that for all sufficiently small $t>0$,
\[\int_{1/t}^{+\infty}\int_{\sigma}|\widehat{f}(tx)|^2 dtdx\leq Ct^{2\alpha+n-1}.\]
\end{enumerate}
\end{corollary}

\section*{Acknowledgements}

Joel E. Restrepo and Durvudkhan Suragan were supported in parts by the Nazarbayev University program 091019CRP2120. Joel E. Restrepo thanks to Colciencias and Universidad de Antioquia (Convocatoria 848 - Programa de estancias postdoctorales 2019) for their support.

\end{document}